\documentclass[12pt]{article}
\usepackage[usenames]{color}
\usepackage{graphicx, subfigure}
\usepackage{amsthm}
\usepackage{amsmath}
\usepackage{amsfonts}

\newtheorem{theorem}{Theorem}[section]
\newtheorem{definition}[theorem]{Definition}
 \newtheorem{lemma}{Lemma}[section]

 \newtheorem{corollary}{Corollary}[section]

\title{Order bornological spaces and order ultrabornological spaces}

\author{Liang Hong \\
Department of Mathematics \\
Robert Morris University   \\
Moon Township, PA 15108, USA\\
Email: hong@rmu.edu}

\begin{document}

\maketitle

\begin{abstract}
Ordered locally convex spaces is an important classes of spaces in the theory of ordered topological vector spaces just as locally convex spaces in the theory of topological vector spaces. Some special classes of ordered locally convex spaces such as order infrabarrelled spaces have been studied. The purpose of this paper is to initiate the study of order bornological spaces and their fundamental properties. In addition, order ultrabornological spaces is also investigated as an important special class of order bornological spaces.
\end{abstract}

\maketitle

\section{Notation and basic concepts}
For notation, terminology and standard results concerning topological vector spaces, we refer to \cite{Bourbaki1}, \cite{Horvath}, \cite{NB} and \cite{Schaefer}; for notation, terminology and standard results concerning Riesz spaces, we refer to \cite{AB2}, \cite{AB1}, \cite{LZ} and \cite{Zaanen}. A partially ordered set $X$ is called a \emph{lattice} if the infimum and supremum of any pair of elements in $X$ exist. A real vector space $X$ is called an \emph{ordered vector space} if its vector space structure is compatible with the order structure in a manner such that
\begin{enumerate}
  \item [(a)]if $x\leq y$, then $x+z\leq y+z$ for any $z\in X$;
  \item [(b)]if $x\leq y$, then $\lambda x\leq \lambda y$ for all $\lambda \geq 0$.
\end{enumerate}
An ordered vector space is called a \emph{Riesz space} (or a \emph{vector lattice}) if it is also a lattice at the same time. A vector subspace of a Riesz space is said to be a \emph{Riesz subspace} if it is closed under the lattice operation $\vee$. A subset $Y$ of a Riesz space $X$ is said to be \emph{solid} if $|x| \leq |y|$ and $y\in Y$ implies $x\in Y$.

Let $V$ be a vector space over a field $K$ and $A$ and $B$ be two subsets of $V$, $A$ is said to \emph{absorb} $B$ if there exists $\alpha>0$ such that $B\subset \lambda A$ for all $\lambda\in K$ such that $|\lambda|\geq \alpha$. The smallest balanced set containing $A$, denoted by $A_b$,  is called the \emph{balanced envelope/hull} of $A$; the smallest convex set containing $A$, denoted by $A_c$, is called the the \emph{convex envelope/hull} of $A$; the smallest balanced convex set containing $A$, denoted by $A_{bc}$, is called the \emph{balanced convex envelope/hull} of $A$.

A topology $\tau$ on a vector space $X$ over a field $K$ is called a \emph{linear topology} or \emph{vector topology} if the addition operation $(x, y)\mapsto x+y$ from $X\times X$ to $X$ and the scalar multiplication operation $(\lambda, x)\mapsto \lambda x$ from $K\times X$ to $X$ are both continuous. A\emph{ topological vector space} $(X, \tau)$ over a field $K$ is a vector space $X$ endowed with a vector topology $\tau$. Unless otherwise stated, all topological vector spaces are assumed to be over $R$. A topological vector space $X$ is said to be \emph{semi-complete} if every Cauchy sequence in $X$ is convergent. A subset $B$ of a topological vector space $(X, \tau)$ is said to be \emph{topologically bounded} or \emph{$\tau$-bounded} if it is absorbed by every neighborhood of zero; as pointed by \cite{Bourbaki1}, this is equivalent to saying for every neighborhood $V$ of zero there exists some $\lambda>0$ such that $\lambda B\subset V$. A topological vector space $(X, \tau)$ is said to be \emph{locally convex} if it has a neighborhood base at zero consisting of convex sets. A locally convex topological vector space is said to be \emph{seminormable} (\emph{normable}) if it can be generated by a single seminorm (norm). The \emph{finest locally convex topology} on a vector space $X$ is the collection of all absorbing, balanced and convex sets of $X$.
An \emph{ordered topological vector space} is an ordered vector space equipped with a compatible vector topology. Let $(X, \tau)$ be an ordered topological vector space. Then the family $\tau_b$ of all the balanced and convex subsets of $X$ each of which absorbs all order bounded subsets of $X$ is called the \emph{order-bound topology} or \emph{order topology} on $X$. An \emph{ordered locally convex space} is an ordered topological vector space that is locally convex at the same time. A \emph{topological Riesz space} is an ordered topological vector space which is a Riesz space at the same time. A topological Riesz space that is locally convex at the same time is called a \emph{locally convex Riesz space}. A vector topology $\tau$ on a Riesz space $X$ is said to be \emph{locally solid} if there exists a $\tau$-neighborhood base at zero consisting of solid sets. A \emph{locally solid Riesz space} is a Riesz space equipped with a locally solid vector topology. A seminorm $\rho$ on a Riesz space $X$ is called a \emph{Riesz seminorm} if $|x|\leq |y|$ implies $\rho(x)\leq \rho(y)$ for any two elements $x, y\in X$. A vector topology on a Riesz space is said to be \emph{locally convex-solid} if it is both locally solid and locally convex. A \emph{locally convex-solid Riesz space} is a Riesz space equipped with a locally convex-solid vector topology.
Following \cite{Wong1}, we say a locally convex-slid Riesz space $E$ is an \emph{order infrabarrelled Riesz space} if each order bornivorous barrel in $E$ is a neighborhood of zero.

All operators in this paper are assumed to be linear; therefore, an operator means a linear operator. Let $T$ be an operator between two ordered topological vector spaces $(E_1, \tau_1)$ and $(E_2, \tau_2)$.
$T$ is said to be a \emph{positive operator} if it carries positive elements to positive elements; it is said to be \emph{regular} if $T$ can be written as the difference of two positive operators; it is said to be  \emph{order bounded} if it carries order bounded sets to order bounded sets; it is said to be \emph{topologically bounded} if it maps $\tau_1$-bounded sets to $\tau_2$-bounded sets; it is said to be   \emph{topologically continuous} if $T^{-1}(O)\in \tau_1$ for every open set $O\in \tau_2$. A topologically continuous operator is topologically bounded.\\

\begin{theorem}[Theorem 2.19 of \cite{AB1}]\label{theorem2.1}
If $(E, \tau)$ is a locally solid Riesz space, then every order bounded subset of $E$ is $\tau$-bounded.
\end{theorem}

\begin{theorem}[Theorem 2.4 of \cite{Hong1}]\label{theorem2.2}
Let $(E, \tau)$ be an ordered topological vector space that has an order bounded $\tau$-neighborhood of zero. Then every $\tau$-bounded subset of $E$ is order bounded.
\end{theorem}

\begin{theorem}
Let $(E, \tau)$ be an ordered topological vector space.
\begin{enumerate}
  \item [(i)]If $E$ is locally solid, its order topology $\tau_b$ is finer than $\tau$, that is, $\tau_b\subset \tau$. If in addition $E$ also has an order bounded $\tau$-neighborhood, then $\tau_b=\tau$.
  \item [(ii)]If $E$ is locally solid and $\tau_b\neq \tau$, then there exists an ordered locally
             convex space $(F, \tau')$ and an order bounded linear operator $T$ from $E$ to $F$ such that $T$ is not topologically continuous.
\end{enumerate}
\end{theorem}

\begin{proof}
\begin{enumerate}
   \item [(i)]The statement follows from Theorem~\ref{theorem2.1} and the fact that order topology is the finest topology in which each order bounded set is topologically bounded (p. 230 of \cite{Schaefer}).
   \item [(ii)]The conclusion follows from (i) and Proposition 2.2 in \cite{Wong3}.
\end{enumerate}
\end{proof}

\section{Order bornological spaces}
Recall that a subset $B$ of a topological vector space $(E, \tau)$ is said to be \emph{bornivorous} if it absorbs all $\tau$-bounded sets in $E$. Below we define an analogous concept for ordered topological vector spaces.

\begin{definition}\label{definition3.1}
\emph{A set $B$ in an ordered topological vector space $(E, \tau)$ over a field $F$ is said to be \emph{order bornivorous} if it absorbs all order bounded sets in $E$, that is, for any order bounded set $D$ in $E$ there exists some $\alpha>0$ such that $D\subset \lambda B$ for all $\lambda \in F$ such that $|\lambda|\geq \alpha$.}
\end{definition}
\noindent \textbf{Remark 1.}
If $\tau$ is locally solid, then every order bornivorous set is bornivorous and hence absorbing; therefore, the gauge $g_B$ of an order bornivorous, balanced and convex set $B$ is a seminorm on $E$. If in addition $B$ is solid, then $g_B$ is seen to be a Riesz seminorm.\\

\noindent \textbf{Remark 2.} If $E$ has an order bounded $\tau$-neighborhood of zero, then every bornivorous set is order bornivorous. Thus, if $E$ has an order bounded $\tau$-neighborhood of zero, then every $\tau$-neighborhood of zero is order bornivorous.\\

\noindent \textbf{Remark 3.} A bornivorous set in an ordered topological vector space need not be order bornivorous, as the next example shows.\\

\noindent \textbf{Example 3.1.} Let $E=R^2$ and $\tau$ be the usual topology on $R^2$. Equip $E$ with the lexicographic ordering, that is, for two points $x=(x_1, x_2)$ and $y=(y_1, y_2)$ in $R^2$ we define $x\leq y$ if
and only if $x_1< y_1$ or else $x_1=y_1$ and $y_2\leq y_2$. It is evident that the closed unit ball $U$ absorbs any $\tau$-bounded sets; hence it is bornivorous. Let $x=(-1, 0)$ and $y=(1, 0)$. Consider the order interval $[x, y]$ which is obviously order bounded.  However, $U$ does not absorb $[x, y]$. Therefore, $U$ is not order bornivorous. \\

\noindent \textbf{Remark 4.} An order bornivorous set in an ordered topological vector space need not be bornivorous. To see this, consider the following example.\\

\noindent \textbf{Example 3.2.} Let $E=D[-\pi, \pi]$ be the space of functions defined on $[-\pi, \pi]$ with continuous first-order derivatives and $\tau$ be the norm topology generated by the sup norm $||x||_{\infty}=\sup_{-\pi\leq t \leq \pi}|x(t)|$. Then $\tau$ is locally convex.
Equip $E$ with the ordering defined as follows: for $x, y\in E$ we define $x\leq y$ if and only if $x(t)\leq y(t)$ and $x'(t)\leq y'(t)$ for all $t\in [-\pi, \pi]$, where $x'$ denotes the derivative of $x$ for $x\in E$. Consider the set $B=\{x\in E\mid x \leq 1\}$. Then $B$ is order bornivorous. However, $B$ is not bornivorous. To see this, we consider the $\tau$-bounded set $A=\{x\in E\mid ||x||_{\infty}\leq 1 \}$. Since $\{\sin kt\}_{k\in N}\in A$ and $\sup_{-\pi\leq t \leq \pi}\{(\sin kt)'\}=\infty$; $B$ cannot absorb $A$. Therefore, $U$ is order bornivorous but not bornivorous.

\begin{definition}\label{definition3.2}
\emph{An ordered locally convex space $(E, \tau)$ is said to be an \emph{order bornological space} if each order bornivorous, balanced and convex set is a neighborhood of zero. An order bornological space that is a Riesz space at the same time is called an \emph{order bornological Riesz space}. }
\end{definition}

\noindent \textbf{Remark.} In general, an order bornological Riesz space is different from  a \emph{bornological Riesz space} which is defined as a locally convex Riesz space that is bornological. The latter class of spaces were studied in \cite{Kawai} and \cite{Wong2}. Example 3.1 shows that there exists an order locally convex space that is bornivorous but not order bornological.

\begin{theorem}\label{theorem3.1}
Suppose $(E, \tau)$ is an ordered locally convex Riesz space. Then the following statements hold.
\begin{enumerate}
  \item [(i)]If $E$ is a normed Riesz space, then it is order bornological. In particular, every Banach lattice is order bornological.
  \item [(ii)]If $E$ is a locally convex-solid order bornological Riesz space, then it is order infrabarreled.
  \item [(iii)] If $E$ is a locally solid order bornological Riesz space, then it is bornological. Conversely, if $E$ is bornological and has an order bounded $\tau$-neighborhood of zero, then it is order bornological.
  \item [(iv)]If $E$ has an order bounded $\tau$-neighborhood of zero, then $E$ equipped with its finest locally convex topology is order bornological.
  \item [(v)] Let $\tau_{ob}$ denote the following collection of subsets of $E$:
            \begin{equation*}
            \quad \quad \quad \{B\in E\mid B \text{  is an order bornivorous, balanced and convex subset of $E$}\}.
            \end{equation*}
            Then $\tau_{ob}$ is a vector topology on $E$ and $(E, \tau_{ob})$ is an order bornological space.
            Moreover, $\tau_{ob}$ is the finest locally convex topology $\tau'$ for $E$ such that
            order bounded sets in $(E, \tau')$ are the same as in $(E, \tau)$.
  \item[(vi)]If $F$ is an ordered vector space and $\{(E_i, \tau_i)\}_{i\in I}$ is a family of locally solid order bornological Riesz spaces, $T_i: E_i\rightarrow F$ is a linear operator for each $i\in I$, and $F$ equipped with the final locally convex topology $\tau_f$ for $\{T_i\}_{i\in I}$ has an order bounded $\tau_f$-neighborhood of zero, then $(F, \tau_f)$ is order bornological.
\end{enumerate}
\end{theorem}

\begin{proof}
\begin{enumerate}
  \item [(i)]Let $(E, \tau)$ be a normed Riesz space, where $\tau$ is the corresponding
             norm topology. Then $\tau$ is locally-convex solid. Let $B$ be an order bornivorous, balanced and convex neighborhood of zero. Since the associated norm $||\cdot||$ is a Riesz norm, for any $r>0$ the ball $B_r=\{x\mid ||x||\leq r\}$ is order bounded. Therefore, these balls are absorbed by $B$. It follows that $B$ is a neighborhood of zero; hence, $E$ is order bornological.
  \item [(ii)] Trivial.
  \item [(iii)] In view of Theorem \ref{theorem2.1}, a bornivorous set in a locally solid Riesz space is order bornivorous. Hence, in a locally solid order bornological Riesz space every bornivorous, balanced and convex set is a neighborhood of zero. Conversely, the hypothesis and Theorem \ref{theorem2.2} imply that an order bornivorous set in $E$ is bornivorous. Hence, every order bornivorous, balanced and convex set in $E$ is a neighborhood of zero.
   \item [(iv)]This follows from (iii) and the fact that $E$ equipped with its finest locally convex topology is
               bornological (p.222 of \cite{Horvath}).
   \item [(v)]The first statement follows from the neighborhood system structure theorem for topological vector spaces (p. I.7 of \cite{Bourbaki1}) and Definition~\ref{definition3.1}. For the second statement, notice that an order bounded set $D$ in $E$ is absorbed by the same collection of balanced and convex sets in $(E, \tau_{ob})$ and $(E, \tau)$.
   \item [(vi)]The hypothesis implies that each $(E_i, \tau_i)$ is a bornological space. Therefore, $(F, \tau_f)$ is also bornological (p. III.12 of \cite{Bourbaki1}). It follows from (iii) that  $(F, \tau_f)$ is order bornological.
   \end{enumerate}
\end{proof}
\noindent \textbf{Remark 1.} The Euclidean norm on $R^2$ is not a Riesz norm when $R^2$ is equipped with the lexicograhpic ordering. Thus, Example 3.1 shows that the conclusion in (i) need not hold if the norm on a Riesz space is not a Riesz norm. A finite dimensional locally convex space is bornological (p. III.12 of \cite{Bourbaki1}); but a finite dimensional ordered locally convex space need not be order bornological, as Example 3.1 shows. Moreover, Example 3.1 shows that a metrizable locally-convex solid Riesz space (even a Fr$\acute{e}$chet lattice) need not be order bornological, although a metrizable
locally convex space is bornological (p. III.2 of \cite{Bourbaki1} or p.222 of \cite{Horvath}).\\

\noindent \textbf{Remark 2}.  The order bornological space $(E, \tau_{ob})$ in (v) is called the \emph{order bornological space associated with $(E, \tau)$}. \\

\begin{corollary}\label{corollary3.1}
Let $(E, \tau)$ be a locally solid order bornological Riesz space and $(F, \tau')$ be the inductive limit, a direct sum or a quotient space of $E$. If a partial order makes $F$ an ordered vector space such that $F$ has an order bounded $\tau'$-neighborhood of zero, then $(F, \tau')$ is order bornological.
\end{corollary}

Operators between ordered topological spaces have received much attention in recent literature; see, for instance,  \cite{EGZ}, \cite{HMZ}, \cite{Hong1}, \cite{Hong2},\cite{Zabeti1} and \cite{Zabeti2} and references therein. Next, we give some fundamental properties of order bounded operators defined on order bornological spaces.
To this aim, we first establish two lemmas. 

\begin{lemma}\label{lemma3.1}
Let $(E, \tau)$ be an ordered topological vector space and $B$ be an order bounded subset of $E$. Then the following statements hold.
\begin{enumerate}
  \item [(i)]The balanced hull $B_b$ of $B$ is order bounded.
  \item [(ii)]The convex hull $B_c$ of $B$ is order bounded.
  \item [(iii)]The convex balanced hull $B_{bc}$ of $B$ is order bounded.
\end{enumerate}
\end{lemma}

\begin{proof}
\begin{enumerate}
  \item [(i)]Since $B$ is order bounded in $E$, there exists an order interval $[x, y]$ in $E$ such that
             $B\subset [x, y]$. For any $a\in B$, some algebra shows that $(-|x|)\wedge (-|y|)\leq \lambda a \leq
             |x|\vee |y|$ for all $\lambda\in F$ such that $|\lambda|\leq 1$, implying that $B_b$ is contained in the order interval$ [(-|x|)\wedge (-|y|), |x|\vee |y|]$. Hence, $B_b$ is order bounded.
  \item [(ii)]The convex hull $B_c$ of $B$ may be written as
                \[
                B_c=\left\{\sum_{i=1}^n \lambda_i a_i \mid  n\in N, \lambda_i\geq 0, \sum_{i=1}^n\lambda_i=1,
                a_i\in B \right\}.
                 \]
                Let $[x, y]$ be an order interval that contains $B$. Then for any $\sum_{i=1}^n \lambda_i a_i\in B_c$ we have
                \[
                x=\sum_{i=1}^n \lambda_i x\leq  \sum_{i=1}^n \lambda_i a_i \leq \sum_{i=1}^n \lambda_i y=y. \]
                Thus, $B_c\in [x, y]$, that is, $B_c$ is order bounded.
  \item [(iii)] This follows immediately from (i) and (ii).
\end{enumerate}
\end{proof}

\begin{lemma}\label{lemma3.2}
Let $(E_1, \tau_1)$ and $(E_2, \tau_2)$ be two ordered topological vector spaces such that $E_2$ has an order bounded $\tau_2$-neighborhood of zero. Suppose $T$ is a linear operator between $E_1$ and $E_2$. Then the following statements are equivalent.
\begin{enumerate}
  \item [(i)]$T$ is order bounded.
  \item [(ii)]$T$ carries order bounded, balanced and convex sets into order bounded, balanced and convex sets.
  \item [(iii)]The inverse image of an order bornivorous, balanced and convex set under $T$ is an order bornivorous, balanced and convex set.
\end{enumerate}
\end{lemma}

\begin{proof}
\begin{enumerate}
  \item []$(i)\Longrightarrow (ii)$. The implication is obvious.
  \item []$(ii)\Longrightarrow (i)$. Let $B$ be an order bounded subset of $E_1$ and $B_{bc}$ be the balanced convex hull of $B$. Then $B_{bc}$ is order bounded by Lemma \ref{lemma3.1}. Since $T(B_{bc})$ is order bounded in $E_2$ and $T(B)\subset T(B_{bc})$, $T(B)$ is also order bounded.
  \item []$(ii)\Longrightarrow (iii)$. Let $A$ be an order bornivorous, balanced and convex set in $E_2$. By the hypothesis, for any order bounded subset $B$ of $E_1$, $T(B)$ is order bounded in $E_2$. Thus, $A$ absorbs $T(B)$, that is, there exists $\alpha>0$ such that $T(B)\subset \lambda A$ for all $\lambda \in F$ such that $|\lambda|\geq \alpha$, or equivalently, $B\subset \lambda T^{-1}(A)$ for all $\lambda \in F$ such that $|\lambda|\geq \alpha$. It follows that $T^{-1}(A)$ absorbs $B$; hence, $T^{-1}(A)$ is order bornivorous. Since convexity and balancedness are preserved under the inverse image of a linear operator, the conclusion follows.
  \item []$(iii)\Longrightarrow (ii)$. Let $B$ be an order bounded disk of $E_1$ and $V$ be
            an order bornivorous, order bounded, balanced and convex $\tau_2$-neighborhood of zero.
            Then $T^{-1}(V)$ is an order bornivorous, balanced and convex set in $E_1$.
            Thus, $T^{-1}(V)$ absorbs $B$, implying that there exists $\alpha>0$ such that
            $T(B)\subset \lambda V$ for all $\lambda \in F$ such that $|\lambda|\geq \alpha$.
            It follows from Theorem \ref{theorem2.2} that $T(B)$ is order bounded in $E_2$.
            Now the implication follows from the fact that convexity and balancedness are both preserved under a linear operator between two vector spaces.
\end{enumerate}
\end{proof}
\noindent \textbf{Remark.} Astute readers may have noticed that only the implication $(iii)\Longrightarrow (ii)$ requires that $E_2$ has an order bounded $\tau_2$-neighborhood of zero.\\

\begin{theorem}\label{theorem3.2}
Let  $(E_1, \tau_1)$ be an ordered locally convex space. Then the following two statements holds.
\begin{enumerate}
  \item [(i)]If $E_1$ is order bornological, then for any ordered locally convex space $(E_2, \tau_2)$ with an order bounded $\tau_2$-neighborhood of zero, each order bounded operator $T: E_1\rightarrow E_2$ is topologically continuous.
  \item [(ii)]If for any ordered locally convex space $(E_2, \tau_2)$, each order bounded operator $T: E_1\rightarrow E_2$ is topologically continuous, then every order bornivorous, balanced, convex and solid set in $E_1$ is a neighborhood of zero.
\end{enumerate}
\end{theorem}

\begin{proof}
\begin{enumerate}
  \item [(i)]Suppose $E_1$ is order bornological. Let $B$ be a balanced and convex neighborhood of zero in $E_2$. Then $B$ is order bornivorous. It follows from Lemma~\ref{lemma3.2} that $T^{-1}(B)$ is an order bornivorous, balanced and convex set in $E_1$, implying that  $T^{-1}(B)$ is a neighborhood of $E_1$.
            This shows that $T$ is topologically continuous.
  \item [(ii)]Let $B$ be an order bornivorous, balanced, convex and solid set in $E_1$ and $g_B$ be the gauge of $B$. Then $g_B$ is a Riesz seminorm on $E_1$. We will denote this seminormed Riesz space by $(X, \tau)$. Then $(X, \tau)$ is a locally convex-solid space with an order bounded $\tau$-neighborhood of zero. Let $A$ be an order bounded subset of $E_1$. Then there exists $\alpha>0$ such that $A\subset \lambda B$ for all $\lambda\in F$ such that $|\lambda|\geq \alpha$. Since the $\tau_1$-closure of $B$ is $\{x\mid g_B(x)\leq 1 \}$ (p.95 of \cite{Horvath}), we have $\lambda B\subset\lambda \{x\mid g_B(x)\leq 1 \}$ for all $\lambda>0$, showing that $A$ is $\tau$-bounded in $X$; hence it is order bounded by Theorem \ref{theorem2.2}. This shows that the identity operator $I: E_1\rightarrow E_2$ is order bounded. It follows from the hypothesis that $B=I^{-1}(B)$ is a neighborhood of $E_1$.
\end{enumerate}
\end{proof}

\section{Order ultrabornological spaces}
As a special class of bornological spaces, ultrabornological spaces were defined and studied in \cite{Grothendieck}. Some basic properties of ultrabornological spaces may be found in \cite{Bourbaki1}. Analogously, we now define and study order ultrabornological spaces which is a special class of order bornological spaces. \\

Let $E$ be a vector space and $B$ be a balanced and convex set in $E$ and $\langle B \rangle$ be the vector subspace generated by $B$. Then $B$ is absorbing in $\langle B \rangle$. Hence, the gauge $g_B$ is a seminorm on $\langle B\rangle$ (p. II.26 of \cite{Bourbaki1}). Following \cite{Bourbaki1}, we will use $E_B$ to denote this seminormed space.

\begin{definition}\label{definition4.1}
\emph{A balanced and convex set $B$ in a topological Riesz space $(E, \tau)$ is said to be  \emph{order infrabornivorous} if it absorbs all order bounded, balanced and convex sets $B$ such that $E_B$ is a Banach lattice.}
\end{definition}

\begin{theorem}\label{theorem4.1}
Let $(E, \tau)$ be a locally solid Riesz space and $B$ be an order bounded, balanced and convex set in $E$. Then the following statements hold. \begin{enumerate}
  \item [(i)]The gauge $g_B$ is a norm, that is, $E_B$ is a normed space. If in addition $B$ is solid in $E_B$, then $E_B$ is a normed Riesz space.
  \item [(ii)]The subspace topology for $E_B$ is weaker than the norm topology on $E_B$ generated by $g_B$.
  \item [(iii)]If in addition $E$ is Hausdorff and $B$ is a solid and semi-complete subset of $E$, then $E_B$ is a Banach lattice.
\end{enumerate}
\end{theorem}

\begin{proof}
\begin{enumerate}
  \item [(i)]Let $x$ be a nonzero element in $E_B$. We only need to show that $g_B(x)=0$ implies $x=0$.
             To this end, consider a neighborhood $V$ of zero in $E$. Since $E$ is locally solid, $B$ is $\tau$-bounded.
             Thus, there exists some $\alpha>0$ such that $\alpha B\subset V$.
             If $x\neq 0$, then we clearly have $x\not \in \alpha B$ which implies $g_B(x)>0$, contradicting the
             hypothesis $g_B(x)=0$. Therefore, $x=0$.
             If in addition $B$ is solid, then $g_B$ is a Riesz norm.
             Therefore, $E_B$ is a normed Riesz space.
  \item [(ii)]Since $E$ is locally solid, $B$ is $\tau$-bounded. Hence, for any $\tau$-neighborhood $V$ of zero, $\alpha B\subset V$ for some $\alpha>0$, implying that $\alpha B\subset V\cap E_B$. This shows that a neighborhood base of the norm topology is contained in a neighborhood base of the subspace topology. Therefore, the conclusion follows.
  \item [(iii)](i) shows that $E_B$ is a normed Riesz space. Since $E$ is Hausdorff and $B$ is semi-complete, $E_B$, as a normed space, is a complete (p. III.8 of \cite{Bourbaki1}).  Therefore, $E_B$ is a Banach lattice.
\end{enumerate}
\end{proof}

\begin{lemma}\label{lemma4.1}
In a Hausdorff locally convex-solid Riesz space $(E, \tau)$ a balanced and convex set $B$ is order infrabornivorous if and only if it absorbs all balanced, compact, convex and solid sets.
\end{lemma}

\begin{proof}
Suppose $B$ absorbs all balanced, compact, convex and solid sets in $E$. Let $A$ be an order bounded, balanced and convex set in $E$ such that $E_A$ is a Banach lattice. Then $B$ absorbs $A$ (Theorem 13.2.2 of \cite{NB}). Conversely, let $A$ be a balanced, compact, convex and solid subset of $E$. Since a compact subset of a topological group is complete (w.r.t the relative uniformity), Theorem \ref{theorem4.1} shows that $E_A$ is a Banach lattice. Therefore, $B$ absorbs $A$.
\end{proof}

\begin{definition}\label{definition4.2}
\emph{A locally convex Riesz space $(E, \tau)$ is said to be an \emph{order ultrabornological space} if each order infrabornivorous, balanced and convex set is a neighborhood of zero.}
\end{definition}
\noindent \textbf{Remark 1.} It is clear from Definition~\ref{definition3.1} and Definition~\ref{definition4.1} that an order bornivorous set is order infrabornivorous. It follows that an order ultrabornological space is an order bornological space.\\

\noindent \textbf{Remark 2.} Let $\tau_{oub}$ denote the following collection of subsets of $E$:
\begin{equation*}
\{B \text{ is infrabornivorous, balanced and convex in $E$ and  $E_B$ is a Banach lattice}\}.
\end{equation*}
Then it is easy to verify that $\tau_{oub}$ satisfies the neighborhood system structure theorem for topological vector space and $(E, \tau_{oub})$ is an order ultrabornological space. We call $(E, \tau_{oub})$ the \emph{order ultrabornological space associated with $(E, \tau)$}. Indeed, $\tau_{oub}$ is the finest locally convex topology $\tau'$ for $E$ such that all order bounded sets $B$ with $E_B$ being a Banach lattice are the same for $(E, \tau')$ and $(E, \tau)$.\\

\begin{lemma}\label{lemma4.2}
Suppose $(E_1, \tau_1)$ and $(E_2, \tau_2)$ are two ordered topological vector spaces such that $E_2$ has an order bounded $\tau_2$-neighborhood. Let $T$ be a linear operator between $E_1$ and $E_2$. Then the following two statements are equivalent.
\begin{enumerate}
  \item [(i)]If $B$ is an order bounded, balanced and convex set in $E_1$ such that $E_B$ is a Banach lattice, then $T(B)$ is order bounded, balanced and convex set in $E_2$.
  \item [(ii)]For any order bornivorous, balanced and convex set $D$ in $E_2$ the inverse image $T^{-1}(D)$ is an order infrabornivorous, balanced and convex set in $E_1$.
\end{enumerate}
\end{lemma}

\begin{proof}
\begin{enumerate}
  \item [](i) $\Longrightarrow$ (ii). Let $D$ be an order bornivorous, balanced and convex set in $E_2$. Let $B$ be an order bounded, balanced and convex set in $E_1$ such that $E_B$ is a Banach lattice. The hypothesis implies $T(B)$ is an order bounded, balanced and convex set in $E_2$. Hence, $D$ absorbs $T(B)$, that is,
                there exists $\alpha>0$ such that $T(B)\subset \lambda (D)$ for all $\lambda\in F$ such that $|\lambda|\geq \alpha$. This is equivalent to saying that $B\subset \lambda T^{-1}(D)$ for all $|\lambda|\geq \alpha$. Therefore, $T^{-1}(D)$ is order infrabornivorous.

  \item [](ii) $\Longrightarrow$ (i). Let $B$ be an order bounded, balanced and convex set in $E_1$ such that $E_B$ is a Banach lattice. In view of Remark 1 of Definition \ref{definition3.1}, any balanced and convex neighborhood $V$ of zero in $E_2$ is order bornivorous. By hypothesis, $T^{-1}(V)$ absorbs $B$,
      or equivalently, there exists $\alpha>0$ such that $T(B)\subset \lambda V$ for all $\lambda\in F$ such that $|\lambda|\geq \alpha$. This shows that $T(B)$ is order bounded.
        Thus, the conclusion follows from the fact that linear operators preserve balancedness and convexity.
\end{enumerate}
\end{proof}
\noindent \textbf{Remark.} It is clear from the proof that only the sufficiency requires that $E_2$ has an order bounded $\tau_2$-neighborhood of zero.

\begin{theorem}\label{theorem4.2}
Suppose $(E_1, \tau)$ is an order ultrabornological space, $(E_2, \tau_2)$ is an ordered locally convex space having an order bounded $\tau_2$-neighborhood, $T$ is a linear operator between $E_1$ and $E_2$, and $B$ is an order bounded, balanced and convex set in $E_1$ such that $E_B$ is a Banach lattice. If
$T(B)$ is an order bounded, balanced and convex set in $E_2$ such that $E_{T(B)}$ is a Banach lattice,
then $T$ is topologically continuous (hence topologically bounded) and order bounded.
\end{theorem}

\begin{proof}
The continuity of $T$ follows from Lemma~\ref{lemma4.2} and Definition~\ref{definition4.2}.
The fact that $T$ is also order bounded follows from Theorem 2.3 of \cite{Hong1}.
\end{proof}

\bibliographystyle{amsplain}

\begin{thebibliography}{n} 



\bibitem{AB2} Aliprantis, C.D. and Burkinshaw, O. (1985). \emph{Positive Operators}, Springer, Berlin, New York.

\bibitem{AB1} Aliprantis, C.D. and Burkinshaw, O. (2003). \emph{Locally Solid Riesz Spaces with Applications to Economics}, Second Edition, American Mathematical Society, Providence, Rhode Island.


\bibitem{Bourbaki1} Bourbaki, N. (1987). \emph{Elements of Mathematics: Topological Vector Spaces}, Chapters 1-5, Springer, Berlin, New York.

\bibitem{EGZ} Erkursun--Ozcan, N., Gezer, N.A.~and Zabeti, O.~(2017). Spaces of $\mu\tau$-Dunford-Pettis and $\mu\tau$-compact operators on locally solid vector lattices. https://arxiv.org/abs/1710.11434.

\bibitem{Grothendieck} Grothendieck, A. (1955). Produits tensoriels topologiques et espaces nucleaires. \emph{Mem. Amer. Math. Soc.}~16, American Mathematical Society, Providence, Rhode Island.

\bibitem{HMZ} Hejazian,S., Mirzavaziri, M.~and Zabeti, O.~(2012). Bounded operators on topological vector spaces and their spectral radii. \emph{Filomat}~26, 1283--1290.

\bibitem{Hong1} Hong, L.~(2016). On order bounded subsets of locally solid Riesz spaces. \emph{Quaestiones Mathematicae}~3, 381--389.

\bibitem{Hong2} Hong, L.~(2018). A note on the relationship between three classes of operators Riesz spaces. Revision submitted to \emph{Annals of West University of Timisoara--Mathematics and Computer Science}. https://arxiv.org/abs/1504.08016.

\bibitem{Horvath} Horv\'{a}th, J. (1966).\emph{Topological Vector Spaces and Distributions}. Vol. I, Addison-Wesley, Reading, Massachusetts.

\bibitem {Kawai} Kawai, I. (1957). Locally convex lattices. \emph{J. Math. Soc. Japan}~9, 281-314.

\bibitem{LZ} Luxemburg, W.A.J. and Zaanen, A.C. (1971). \emph{Riesz Spaces, I}, North-Holland, Amsterdam.

\bibitem{NB} Narici, L. and Beckenstein, E. (2011). \emph{Topological Vector Spaces}, Second Edition, CRC Press, Boca Raton.


\bibitem{Schaefer} Schaefer, H.H., (1974). \emph{Toploogical Vector Spaces}, Springer, Berlin, New York.


\bibitem{Wong1} Wong, Y. (1969). Order-infrabarrelled Riesz spaces. \emph{Math. Ann.}~183, 17-32.

\bibitem{Wong2} Wong, Y. (1970). The order-bound topology on Riesz spaces. \emph{Proc. Camb. phil. Soc.}~67, 587-593.

\bibitem{Wong3} Wong, Y. (1972). The order-bound topology. \emph{Proc. Camb. phil. Soc.}~71, 321-327.

\bibitem{Zaanen} Zaanen, A. (1997).\emph{Introduction to Operator Theory in Riesz Spaces}, Springer, Berlin, New York.

\bibitem{Zabeti1} Zabeti, O. (2011). Topological algebras of bounded operators on topological vector spaces. \emph{J. Adv. Res. Pure Math}~3, 22--26

\bibitem{Zabeti2} Zabeti, O. (2017). Topological algebras of locally solid vector subspaces of order bounded operators. \emph{Positivity}~21, 1253--1259.


\end{thebibliography}

\end{document}